\documentclass[12pt]{article}
\usepackage{e-jc}
\usepackage{amsmath}
\usepackage{amssymb}
\usepackage{fullpage}
\usepackage{url}

\begin{document}

\date{\small Mathematics Subject Classification: 05D10}

% This is a file of macros and definitions that may come up in
% ANY LATEX paper.  Typically I'll use this file and some other file in the
% directory of the paper, this one for general math things, that one for
% things specific to that paper.
%
% font's used and general paper things.
%\font\tenrm=cmr10
%\font\ninerm=cmr9
%\font\eightrm=cmr8
%\font\sevenrm=cmr7
% \font\title=cmbx10 scaled \magstep1 % extra big title font
%\font\ss=cmss10 % used by \proof
%\font\smallcaps=cmcsc10 % used to label Theorems, etc.
% imhibit black bars on overflows
%
%\overfullrule=0pt
%
% today's date
%
%
% English words that I always italizice in papers.
% Some words that appear in math mode alot that I wasn roman
%

\def\poly{\mathrm{poly}}
\def\min{\mathrm{min}}
\newcommand{\diff}{{\rm diff }}
\newcommand{\thresh}{t}

\newcommand{\PP}{{\rm P}}
\newcommand{\NP}{{\rm NP}}
\newcommand{\eps}{\epsilon}
\newcommand{\st}{\mathrel{:}}
\newcommand{\kstar}{{\textstyle *}}
\newcommand{\bits}[1]{\{0,1\}^{{#1}}}
\newcommand{\bit}{\{0,1\}}
\newcommand{\nat}{{\sf N}}
\newcommand{\Z}{{\sf Z}}
\newcommand{\C}{{\sf C}}
\newcommand{\R}{{\sf R}}
\newcommand{\rat}{{\sf Q}}
\newcommand{\BPP}{{\rm BPP}}
\newcommand{\DTIME}{{\rm DTIME}}
\newcommand{\HARD}{{\rm HARD}}
\newcommand{\into}{\rightarrow}
\renewcommand{\AE}{\forall^\infty}
\newcommand{\IO}{\exists^\infty}
\newcommand{\ep}{\ep}
\newcommand{\es}{\emptyset}
\newcommand{\isom}{\simeq}
\newcommand{\nisom}{\not\simeq}
\newcommand{\lf}{\left\lfloor}
\newcommand{\rf}{\right\rfloor}
\newcommand{\lc}{\left\lceil}
\newcommand{\rc}{\right\rceil}
\newcommand{\Ceil}[1]{\left\lceil {#1}\right\rceil}
\newcommand{\ceil}[1]{\left\lceil {#1}\right\rceil}
\newcommand{\floor}[1]{\left\lfloor{#1}\right\rfloor}
\newcommand{\nth}{n^{th}}
%
% definitions with macros
%
\newcommand{\inter}{\cap}
\newcommand{\union}{\cup}
\newcommand{\sig}[1]{\sigma_{#1} }
\newcommand{\s}[1]{\s_{#1}}
\newcommand{\dash}{\hbox{-}}
\newcommand{\infinity}{\infty}
\newcommand{\ie}{\hbox{ i.e.  }}
\newcommand{\eg}{\hbox{ e.g.  }}

\newcounter{savenumi}
\newenvironment{savenumerate}{\begin{enumerate}
\setcounter{enumi}{\value{savenumi}}}{\end{enumerate}
\setcounter{savenumi}{\value{enumi}}}

\newtheorem{theoremfoo}{Theorem}[section] %by chapter in report style
\newenvironment{theorem}{\pagebreak[1]\begin{theoremfoo}}{\end{theoremfoo}}

\newtheorem{lemmafoo}[theoremfoo]{Lemma}
\newenvironment{lemma}{\pagebreak[1]\begin{lemmafoo}}{\end{lemmafoo}}

\newtheorem{corollaryfoo}[theoremfoo]{Corollary}
\newenvironment{corollary}{\pagebreak[1]\begin{corollaryfoo}}{\end{corollaryfoo}}

\newcommand{\fig}[1] %usage:\fig{file}
{
 \begin{figure}
 \begin{center}
 \input{#1}
 \end{center}
 \end{figure}
}

\newtheorem{notefoo}[theoremfoo]{Note}
\newenvironment{note}{\pagebreak[1]\begin{notefoo}\rm}{\end{notefoo}}

\newtheorem{nttn}[theoremfoo]{Notation}
\newenvironment{notation}{\pagebreak[1]\begin{nttn}\rm}{\end{nttn}}

\newtheorem{dfntn}[theoremfoo]{Definition}
\newenvironment{definition}{\pagebreak[1]\begin{dfntn}\rm}{\end{dfntn}}

\newenvironment{proof}
    {\pagebreak[1]{\narrower\noindent {\bf Proof:\quad\nopagebreak}}}{\QED}
\newenvironment{sketch}
    {\pagebreak[1]{\narrower\noindent {\bf Proof sketch:\quad\nopagebreak}}}{\QED}
\newenvironment{comment}{\penalty -50 $(*$\nolinebreak\ }{\nolinebreak $*)$\linebreak[1]\ }

\newenvironment{algorithm}[1]{\bigskip\noindent ALGORITHM~#1\renewcommand{\theenumii}{\arabic{enumii}}\renewcommand{\labelenumii}{Step \theenumii :}\begin{enumerate}}{\end{enumerate}END OF ALGORITHM\bigskip}

\newcommand{\yyskip}{\penalty-50\vskip 5pt plus 3pt minus 2pt}
\newcommand{\blackslug}{\hbox{\hskip 1pt
        \vrule width 4pt height 8pt depth 1.5pt\hskip 1pt}}
\newcommand{\QED}{{\penalty10000\parindent 0pt\penalty10000
        \hskip 8 pt\nolinebreak\blackslug\hfill\lower 8.5pt\null}
        \par\yyskip\pagebreak[1]}

\newcommand{\BBB}{{\penalty10000\parindent 0pt\penalty10000
        \hskip 8 pt\nolinebreak\hbox{\ }\hfill\lower 8.5pt\null}
        \par\yyskip\pagebreak[1]}
     
\newcommand{\PYI}{CCR-8958528}
\newtheorem{factfoo}[theoremfoo]{Fact}
\newenvironment{fact}{\pagebreak[1]\begin{factfoo}}{\end{factfoo}}
\newenvironment{acknowledgments}{\par\vskip 6pt\footnotesize Acknowledgments.}{\par}

%%%%%%%%%%%%%%%%%%%%%%%%%%%%%%%%%%%%%%%%%%%%%%%%%%%%%%%%%%%%%%%%%%%%%%%%%%%%

%%%%%%%%%% 
%%%%%%%%%%  Jim's construction macros.
%%%%%%%%%% 

%%% The follow are macros for displaying block structured programs and
%%% constructions.  Basically, they are dressed up lists, like the
%%% enumerate and itemize environments.  Use  the construction environment
%%% for the outermost ``list'' of instruction and for ``sublists'' of
%%% instructions use the block environment.  E.g., 
%%%     
%%% \begin{construction}
%%%   \item {\bf Program for} $M_{p,a}$. 
%%%   \begin{block}
%%%     \item Input $x$.
%%%     \item Instructions.  Instructions.  Instructions.  Instructions.
%%% 	      Instructions.  Instructions.  Instructions.  Instructions.
%%%     \begin{block}
%%%        \item More instructions.  More instructions.  More
%%% 		instructions.   More instructions.  
%%%        \item More instructions.  More instructions.  More
%%% 		instructions.   More instructions.  
%%%     \end{block}
%%%     \item Instructions.  Instructions.  Instructions.  Instructions.
%%%     \item Instructions.  Instructions.  Instructions.  Instructions.
%%%   \end{block}
%%%   \item {\bf End program for} $M_{p,a}$.
%%% \end{construction}

\newenvironment{construction}{\bigbreak\begin{block}}{\end{block}
    \bigbreak}

\newenvironment{block}{\begin{list}{\hbox{}}{\leftmargin 1em
    \itemindent -1em \topsep 0pt \itemsep 0pt \partopsep 0pt}}{\end{list}}

%%% If you want to label the statements/blocks in your construction,  use
%%% the lblock environment instead of the block environment and for each
%%% item macro, use \item[YOUR_LABEL].  Note that labels are
%%% do-it-yourself.
%%% 
%%% Note that in the following, the basic indentation of an lblock at
%%% level i of nesting (i>0) is (\dimen15 + i * \dimen16).  The default 
%%% value of both \dimen15 and \dimen16 is  0.75em.

\dimen15=0.75em
\dimen16=0.75em

\newcommand{\lblocklabel}[1]{\rlap{#1}\hss}

\newenvironment{lblock}{\begin{list}{}{\advance\dimen15 by \dimen16
    \leftmargin \dimen15
    \itemindent -1em
    \topsep 0pt
    \labelwidth 0pt
    \labelsep \leftmargin
    \itemsep 0pt
    \let\makelabel\lblocklabel
    \partopsep 0pt}}{\end{list}}

%%% The lconstruction is an alternative to the construction environment
%%% which lets you temporarily change the values of \dimen15 and \dimen16.

\newenvironment{lconstruction}[2]{\dimen15=#1 \dimen16=#2
  \bigbreak\begin{block}}{\end{block}\bigbreak}

\newcommand{\Comment}[1]{{\sl ($\ast$\  #1\ $\ast$)}}

%%%%%%%%%% 
%%%%%%%%%% End of Jim's construction macros.
%%%%%%%%%% 

\hyphenation{pro-gres-sions}

\newcommand{\lam}{\lambda}
\renewcommand{\ep}{\epsilon}
\newcommand{\Abar}{\overline{A}}
\newcommand{\Tbar}{\overline{T}}
\newcommand{\sm}{{\rm small}}
\newcommand{\bi}{{\rm big}}
\newcommand{\ts}{t^{\sm}}
\newcommand{\tb}{t^{\bi}}
\newcommand{\tsj}{t^{\sm,j}}
\newcommand{\tbj}{t^{\bi,j}}
\newcommand{\Bs}{B^{\sm}}
\newcommand{\Bb}{B^{\bi}}

%\centerline{\bf A Survey of Lower Bounds on the van der Waerden Numbers $W(k,2)$:}
%\centerline{\bf Randomized-Constructive and Deterministic-Constructive}
%\centerline{\bf By William Gasarch (gasarch@cs.umd.edu)}
%\centerline{\bf and Bernhard Haeupler (haeupler@mit.edu)}

\setlength{\textheight}{8.7in}
\setlength{\textwidth}{6.3in}

\renewcommand{\labelitemii}{\labelitemi}
\renewcommand{\labelitemiii}{\labelitemi}

\title{Lower Bounds on van der Waerden Numbers:\\ \large Randomized- and Deterministic-Constructive}

\author{William Gasarch\\
\small Department of Computer Science\\[-0.8ex]
\small University of Maryland at College Park\\[-0.8ex]
\small College Park, MD 20742, USA\\
\small \texttt{gasarch@cs.umd.edu}\\
\and
Bernhard Haeupler\\
\small CSAIL\\[-0.8ex]
\small Massachusetts Institute of Technology\\[-0.8ex]
\small Cambridge, MA 02130, USA\\
\small \texttt{haeupler@mit.edu}\\
}

\maketitle

\begin{abstract}
The van der Waerden number $W(k,2)$ is the smallest integer $n$ such 
that every $2$-coloring of 1 to $n$ has a monochromatic arithmetic 
progression of length $k$. The existence of such an $n$ for any $k$ is 
due to van der Waerden but known upper bounds on $W(k,2)$ are enormous. 
Much effort was put into developing lower bounds on $W(k,2)$.
Most of these lower bound proofs employ the probabilistic method often in combination with the 
Lov{\'a}sz Local Lemma. While these proofs show the existence of a 
$2$-coloring that has no monochromatic arithmetic 
progression of length $k$ they provide no efficient algorithm to find such a 
coloring. These kind of proofs are often informally called {\it 
nonconstructive} in contrast to {\it constructive} proofs that provide 
an efficient algorithm.

This paper clarifies these notions and gives definitions for 
deterministic- and randomized-constructive proofs as different 
types of constructive proofs. We then survey the literature on lower 
bounds on $W(k,2)$ in this light. We show how known nonconstructive 
lower bound proofs based on the Lov{\'a}sz Local Lemma can be made 
randomized-constructive using the recent algorithms of Moser and 
Tardos. We also use a derandomization of Chandrasekaran, Goyal and 
Haeupler to transform these proofs into deterministic-constructive proofs. 
We provide greatly simplified and fully self-contained proofs and descriptions for 
these algorithms. 
\end{abstract}

\section{Introduction}

\begin{notation}
Let $[n]=\{1,\ldots,n\}$ and $\nat^+ = \{1,2,\ldots\}$. If $k\in \nat^+$ then a $k$-AP means an arithmetic progression
of size $k$, i.e., $k$ numbers of the form $\{a,a+d,\ldots,a+(k-1)d\}$ with $a,d\in \nat^+$.
\end{notation}

Recall van der Waerden's theorem:

\begin{theorem}\label{th:vdw}
For every $k\ge 1$ and $c\ge 1$ there exists $W$ such that
for every $c$-coloring $COL:[W]\into [c]$
there exists a monochromatic $k$-AP, i.e. there are $a,d\in \nat^+$, such that
$$COL(a)=COL(a+d)= \cdots = COL(a+(k-1)d).$$
\end{theorem}

\begin{definition}
Let $k,c,n\in \nat$ and let $COL:[n]\into [c]$. We say that {\it $COL$ is a $(k,c)$-proper coloring of $[n]$} if there is no monochromatic $k$-AP in $[n]$. We denote with $W(k,c)$ the least $W$ such that van der Waerden's theorem holds with these values of $k,c$ and $W$, i.e., the least $W$ such that there exists no proper coloring of $[W]$.
\end{definition}

The first proof of Theorem \ref{th:vdw} was due to van der Waerden~\cite{VDW}. 
The bounds on $W(k,c)$ were (to quote Graham, Rothchild, and Spencer~\cite{GRS})
EEEENORMOUS. Formally they were not primitive recursive. 
The proof is purely combinatorial. Shelah~\cite{VDWs} 
gave primitive recursive bounds with a purely combinatorial proof. The best bound is due to Gowers~\cite{Gowers} who used rather hard mathematics to obtain 

%$$
%{\Large  W(k,c) \le 2^{2^{c}{^{2^{2^{k+9}}}}}}.
%$$

$$
W(k,c) \le 2^{2^{c}{^{2^{2^{k+9}}}}}.
$$

In this paper we survey lower bounds for van der Waerden numbers. Some of the bounds are obtained by probabilistic proofs.
Since such proofs do not produce an actual coloring they are often called, informally, {\it nonconstructive}.
However, since all of the objects involved are finite, one could
(in principle) enumerate all of the colorings until one with the correct properties is found. We do not object to the term
{\it nonconstructive}; however, we wish to clarify it.
To this end we formally define two types of constructive proofs.
We only define these notions for proofs of lower bounds on $W(k,c)$.
It would be easy to define constructive proofs in general; however, we want to keep our presentation simple and focused.

\begin{definition}
A proof that $W(k,c)\ge f(k,c)$ is {\it deterministic-constructive} if it presents an algorithm
that will, for all $k,c$, produce a proper $c$-coloring of $[f(k,c)]$
in time polynomial in $f(k,c)$.
\end{definition}

Some of the nonconstructive techniques yield a
randomized algorithm that, with high probability, will
produce a proper coloring in polynomial time. These seem to us to be
different from truly nonconstructive techniques.
Hence we define a notion of randomized-constructive.

\begin{definition}
A proof that $W(k,c)\ge f(k,c)$ is {\it randomized-constructive} if it presents a randomized algorithm
that will, for all $k,c$, 
\begin{itemize}
\item
always produce either a proper $c$-coloring or the statement {\bf I HAVE FAILED!},
\item
with probability $\ge 2/3$ produce a proper $c$-coloring, and
\item
terminate in time polynomial in $f(k,c)$.
\end{itemize}
\end{definition}

\begin{note}~
\begin{enumerate}
\item
The success probability can be increased through standard amplification by repeating the algorithm (say) $f(k,c)$ times to make the probability of success $1 - \frac{1}{3^{f(k,c)}}$ or even higher. The required explicitly declared one-sided error makes it furthermore possible to transform each randomized-constructive proof into a Las Vegas algorithm that always outputs a proper $c$-coloring in expected polynomial time. 
\item
Similar probabilistic proofs of lower bounds for (off-diagonal) Ramsey Numbers~\cite{GRS,LLLnew} are neither deterministic-constructive nor randomized-constructive. The reason for this is that no polynomial time algorithm for detecting a failure (i.e., finding a large clique or independent set) is known. This makes randomized 
algorithms such as the ones by 
Haeupler, Saha, and Srinivasan~\cite{LLLnew} inherently Monte Carlo 
algorithms that cannot be made randomized-constructive.

\item
Work of Wigderson et al.~\cite{IW,NW} on derandomization shows that, under widely believed but elusive to prove hardness assumptions, randomness does not help algorithmically - or more formally that $\PP = \BPP$. In this case the above two notions of randomized-constructive and deterministic-constructive would coincide. 
\end{enumerate}
\end{note}

We present the following lower bounds:

\begin{enumerate}
\item
$W(k,2) \ge \sqrt{\frac{k}{3}} 2^{(k-1)/2}$ by a randomized-constructive proof.
This is an easy and known application of the probabilistic method of Erd{\"o}s and Rado~\cite{ErdosRado}.
This result is usually presented as being nonconstructive.

\item
$W(k,2)\ge  \sqrt k 2^{(k-1)/2}$ by a deterministic-constructive proof.
This is an easy derandomization of the Erd{\"o}s-Rado  lower bound using the method of conditional expectations
of Erd{\"o}s and Selfridge~\cite{condexp}. 
It is likely known though we have never seen it stated.

\item
If $p$ is prime then $W(p+1,2) \ge p(2^p-1)$ by a deterministic-constructive proof. 
Berlekamp~\cite{vdwprimes} proved this; however, our presentation follows that of 
Graham et al~\cite{GRS}. 
Berlekamp actually proved $W(p+1,2) \ge p2^p$. He also has lower bounds if $k$ is a prime power
and $c$ is any number.
Using a hard result from number theory~\cite{BHP} we obtain as a corollary that, for all but a finite number of $k$,
$W(k,2) \ge (k-k^{0.525})(2^{k-k^{0.525}}-1).$

\item
$W(k,2) \ge \frac{2^{(k-1)}}{4k}$ by a randomized-constructive proof. 
The nonconstructive version of this bound is implied by 
the Lov{\'a}sz Local Lemma~\cite{LLL} and by Szab{\'o}'s result~\cite{szabolower} (explained below).
The randomized-constructive proof is an application of 
Moser's~\cite{moserstoc} algorithmic proof of the Lov{\'a}sz Local Lemma.
Our presentation is based on Moser's STOC presentation~\cite{mosertalk} 
in which he sketched a Kolmogorov complexity based proof 
that differed significantly from the conference paper~\cite{moserstoc}. 
Later Moser and Tardos wrote a sequel making the general Lov{\'a}sz Local Lemma 
(with the optimal constants) constructive~\cite{mosertardos}. 
Schweitzer had, independently, used Kolmogorov complexity to obtain lower bounds on $W(k,c)$~\cite{vdwkolm}.

\item
For all $\eps>0$, for all $k\in\nat^+$, $W(k,2)\ge \frac{2^{(k-1)(1 - \eps)}}{ek}$ by a deterministic-constructive proof. More precisely we give a deterministic algorithm that, given $k$ and $\eps$, always outputs a proper coloring of $[\frac{2^{(k-1)(1 - \eps)}}{ek}]$ in time $2^{O(k/\eps)}$ which is polynomial in the output size for any constant $\eps >0$. This result is an application of a derandomization of the Moser-Tardos algorithm for the Lov{\'a}sz Local Lemma given by 
Chandrasekaran, Goyal and B.~Haeupler~\cite{LLLd}.
We present a simplified, short and completely self-contained proof.

\item The Lov{\'a}sz Local Lemma algorithm by Moser and Tardos~\cite{mosertardos} can be used to obtain $W(k,2)\ge \frac{2^{(k-1)}}{ek}$ by a randomized-constructive proof matching the best nonconstructive bound directly achievable via the Lov{\'a}sz Local Lemma (see \cite{GRS}). We show $W(k,2) \ge \frac{2^{(k-1)}}{ek} - 1$ as a simple corollary of our deterministic-constructive proof. 

%\item
%We give evidence, from the field of derandomization, that deterministic-constructive and randomized-constructive are the same. This only applies for van der Waerden lower bounds (and other similar theorems) but e.g. not for Ramsey's theorem.

\end{enumerate}

\begin{note}~
\begin{enumerate}
\item
The best known (asymptotic) lower bound on $W(k,2)$ is due to Szab{\'o}~\cite{szabolower}:
$$
\forall \epsilon>0, \ \forall \hbox{ large $k$}: \ W(k,2) \ge \frac{2^k}{k^\epsilon}.
$$
The proof is involved, relies on the Lov{\'a}sz Local Lemma and additionally exploits the structure of $k$-APs that almost all $k$-AP are almost disjoint (i.e., intersect in at most one number). 
While the original proof is nonconstructive it can be made constructive 
using the methods of some recent papers~\cite{LLLd,LLLnew,mosertardos}.

\item There is no analog of Szab{\'o}'s bound for $c\ge 3$ colors known. In contrast to this the techniques presented here directly extend to give lower bounds on multi-color van der Waerden numbers of the form $W(k,c) \geq \frac{c^{(k-1)}}{ek}$ for any integer $c\geq 2$. 
\item
The techniques used to prove the results mentioned in items 1,2,3,5, and 6 can be modified 
to get lower bounds for variants of van der Waerden numbers such as Gallai-Witt numbers (multi-dimensional van der Warden Numbers)~\cite{radogerman,radoenglish} (see also \cite{VDWbook,GRS}), and some polynomial van der Waerden numbers~\cite{pvdw,pvdww} (see also \cite{VDWbook}).
\end{enumerate}
\end{note}

We use the following easy lemmas throughout the paper.

\begin{lemma}\label{le:numof}
Let $k,n\in\nat^+$.
\begin{enumerate}
\item
Given a $k$-AP of $[n]$ the number of $k$-AP's that intersect it is less than $kn$.
\item
The number of $k$-AP's of $[n]$ is less than $n^2/k$.
\end{enumerate}
\end{lemma}

\begin{proof}

\noindent
1.) We first bound how many $k$-AP's contain a fixed number $x \in [n]$.
Let $1 \le i\le k$. 
If $x$ is the $i^{th}$ element of some $k$-AP then in order for this
$k$-AP to be contained in $[n]$ its step width $d$ has to obey: $1 \ \le \ x - (i-1)d \ \ \ \mbox{ and } \ \ \ x+(k-i)d \ \le \ n.$

We assume for simplicity that $k$ is even (the odd case is nearly identical).
Once $i$ and $d$ are fixed, the $k$-AP is determined. We sum over
all possibilities of $i$ while assuming the second bound on $d$ for all $i \leq k/2$
and the first bound for $i > k/2$. This gives us the following
upper bound on the number of $k$-APs going through a fixed $x$:

$$\sum_{i=1}^{k/2} \frac{n-x}{k-i} \ + \; \sum_{i=k/2+1}^{k} \frac{x-1}{i-1} \ = \ (n-x)\sum_{i=1}^{k/2} \frac{1}{k-i} \ + \ (x-1)\sum_{i=k/2+1}^{k} \frac{1}{i-1} \ =$$

$$= \ (n-x+x-1)\sum_{i=k/2+1}^{k} \frac{1}{i-1} \ \leq \ n - 1.$$

Here the last inequality follows from $\sum_{i=k/2}^{k} \frac{1}{i-1}\le 1$ which can be easily shown by induction.
Using this upper bound we get that the number of $k$-AP's that intersect a given $k$-AP is at most $k(n-1) < kn$.

\bigskip

\noindent
2.) If a $k$-AP has starting point $a$ then then $a+(k-1)d \le n$, so $d \le \frac{n-a}{k-1}.$
Hence, for any $a\in [n]$, there are at most $\frac{n-a}{k-1}$ $k$-AP's that start with $a$.
The total number of $k$-AP's in $[n]$ is thus bounded by 

$$\sum_{a=1}^{n-1}\frac{n-a}{k-1}=\frac{1}{k-1}\sum_{a=1}^{n-1} n-a = \frac{n(n-1)}{2(k-1)} < \frac{n^2}{k}.$$
\end{proof}

\section{A Simple Randomized-Constructive Lower Bound}\label{sec:simplerconstr}

\begin{theorem}\label{thm:simplerconstructive}
$W(k,2)\ge \sqrt{\frac{k}{3}} 2^{(k-1)/2}$ by a randomized-constructive proof.
\end{theorem}

\begin{proof}
We first present the classic nonconstructive proof and then show how to make it into a randomized-constructive proof.

Let $n = \sqrt{\frac{k}{3}} k 2^{(k-1)/2}$. 
Color each number $x$ from $1$ to $n$ by flipping a fair coin.
If the coin is heads then color $x$ with 0,  if the coin is tails then color $x$ with 1.
Let $p$ be the probability that there is a monochromatic $k$-AP. We will show that $p<1$ and hence there is some choice of coin flips that leads to a proper 2-coloring of $[n]$.

By Lemma \ref{le:numof} the number of $k$-AP's is bounded by $n^2/k$. Because of the random choice of colors each $k$-AP becomes monochromatic with probability exactly $2^{-(k-1)}$ and a simple union bound over all $k$-AP's gives:

$$p \le (n^2/k) 2^{-(k-1)} = \frac{n^2}{k2^{(k-1)}}.$$

Looking ahead to making this proof randomized-constructive we want this probability to be at most $1/3$. We show that this is implied by our choice of $n$.

$$\frac{n^2}{k2^{k-1}} \le 1/3$$

$$3n^2 \le k2^{k-1}$$

$$\sqrt 3 n \le \sqrt k 2^{(k-1)/2}$$

$$n \le \sqrt{\frac{k}{3}} 2^{(k-1)/2}.$$

We now present a randomized algorithm that produces (with high probability)
a proper coloring and admits its failure when it does not.

\begin{enumerate}
\item
Get input $k$ and let $n=\sqrt{\frac{k}{3}} 2^{(k-1)/2}$.
\item
Use $n$ random bits to color $[n]$.
\item
Check all $k$-APs of $[n]$ to see if any are monochromatic. (by Lemma~\ref{le:numof} there are at most $n^2/k$ different $k$-APs to check, 
so this takes $O(n^2)$ time). If none are monochromatic then the 
coloring is proper and we output it. Else output {\bf I HAVE FAILED!}.
\end{enumerate}

By the above calculations the probability of success is $\ge 2/3$.
By comments made in the algorithm it runs in polynomial time.
\end{proof}

\section{A Simple Deterministic-Constructive Proof}

\begin{theorem}
$W(k,2)\ge \sqrt k 2^{(k-1)/2}$ by a deterministic-constructive proof.
\end{theorem}

\begin{proof}
We derandomize the algorithm from Section \ref{sec:simplerconstr} using the method of conditional probabilities~\cite{LLL}. Let $n< {\sqrt k 2^{(k-1)/2}}$ and $X$ be the set of all arithmetic progressions of length $k$ that are contained in $[n]$.

Let $f:\R^n \into \R$ be defined by

$$f(x_1,\ldots,x_n) = \sum_{s\in X} ( \prod_{i\in s} x_i + \prod_{i\in s} (1-x_i) ).$$

We will color $[n]$ with 0's and 1's. Assume we have such a coloring and that $x_i$ is the color of $i$. When $x_i$ is set to $1/2$ that means that we have not colored it yet. Note that $f(x_1,\ldots,x_n)$ gives exactly the expected number of monochromatic $k$-AP's when each number $i$ gets colored independently with probability $P(i \textrm{ is colored }1) = x_i$. Thus a coloring has a monochromatic $k$-AP iff $f(x_1,\ldots,x_n) \ge 1$. We will color $[n]$ such that $f(x_1,\ldots,x_n) < 1$.

Note that 
\[
\begin{array}{rl}
f(1/2,\ldots,1/2)= \sum_{s\in X} ( \prod_{i\in s} 1/2 + \prod_{i\in s} 1/2) & =  \sum_{s\in X} ((1/2)^k  + (1/2)^k)\cr 
                                                                            & = \sum_{s\in X} (1/2)^{k-1}\cr
                                                                            &\le n^2/(k 2^{k-1})\cr
\end{array}
\]

We need this to be $<1$. We set this $<1$ which will derive what $n$ has to be.

\[
\begin{array}{rl}
n^2/(k 2^{k-1}) & < 1 \cr
n^2< & k 2^{k-1}\cr
n< & \sqrt{k}2^{(k-1)/2}\cr
\end{array}
\]

We now present a deterministic algorithm:

\begin{enumerate}
\item
Let $x_1=x_2=\cdots=x_n=1/2$. By Lemma~\ref{le:numof} the number of $k$-AP's is $\le n^2/k$.
By the above calculation $f(x_1,\ldots,x_n)<1$.
\item
For $i=1$ to $n$ do the following.
When we color $i$ we already have $1,2,\ldots,i-1$ colored.
Let the colors be $c_1,\ldots,c_{i-1}$. Hence our function now looks like,
leaving the color of $i$ a variable,
$f(c_1,\ldots,c_{i-1},z,1/2,\ldots,1/2)$.
This is a linear function of $z$. We know inductively that if $z=1/2$ then the value is
$<1$. If the coefficient of $z$ is positive then color $i$ 0.
If the       coefficient of $z$ is negative then color $i$ 1.
In either case this will ensure that
$$f(c_1,\ldots,c_i,1/2,\ldots,1/2) \le f(c_1,\ldots,c_{i-1},1/2,\ldots,1/2) < 1.$$
\end{enumerate}

At the end we have $f(x_1,\ldots,x_n)<1$ and hence we have a proper 2-coloring.
It is easy to see that this algorithms runs in time polynomial in $n$.
\end{proof}

\section{An Algebraic Lower Bound}

We will need the following facts.

\begin{fact}\label{fa:known}~
Let $p\in\nat$ (not necessarily a prime).
\begin{enumerate}
\item\label{fa:knownvec}
There is a unique (up to isomorphism) finite field
of size $2^p$. We denote this field by $F_{2^p}$. 
$F_{2^p}$ can be represented by $F_2[x]/<i(x)>$ where $i$ is an irreducible polynomial
of degree $p$ in $F_2[x]$.
$F_{2^p}$ can be viewed as a vector space of dimension $p$ over $F_2$.
The basis of this vectors space is (the equivalence classes of) $1,x,x^2,\ldots,x^{p-1}$.
\item\label{fa:knowngen}
The group $F_{2^p}-\{0\}$ under multiplication is isomorphic to the cyclic
group on $2^p-1$ elements. Hence it has a generator $g$ such that
$$F_{2^p}-\{0\} = \{g,g^2,g^3,\ldots,g^{2^p-1}\}.$$
This generator can be found in time polynomial in $2^p$.
\item \label{fa:knownkey}
Assume $p$ is prime.
Let $g$ be a generator of $F_{2^p}$, and $\beta=g^d$ where $1\le d< 2^p-1$.
We do all arithmetic in $F_{2^p}$.
Let $P$ be a nonzero polynomial of degree $\le p-1$, with coefficients in $\{0,1,2,\ldots,2^p-1\}$.
Then $P(g)\ne 0$ and $P(\beta)\ne 0$.
\end{enumerate}
\end{fact}

\begin{proof} The first two facts are well known and hence we omit the proof. To see the third fact note that $F_{2^p}$ can be viewed as a vector space of dimension $p$ over $F_2$. There can be no field strictly between $F_2$ and $F_{2^p}$: if there was then its dimension as a vector space over $F_2$ would be a proper divisor of $p$. For any $a \in F_{2^p} - F_2$ we get now that $F_2(a)$ is $F_{2^p}$ because it would otherwise be a field strictly between $F_2$ and $F_{2^p}$.
Hence the minimal polynomial of $a$ in $F_2[X]$, which we denote $Q$, has degree $p$. Let $P$ be a nonzero polynomial in $F_2[X]$ of degree at most $p-1$.
If $P(a)=0$ then $P$ has to be a multiple of $Q$. Since $P$ has degree $\le p-1$ and $Q$ has degree $p$, this is impossible. 
Hence $P(a)\ne 0$. This applies to $a=g$ and to $a=g^d$ with $1\le d \le 2^p - 2$. (Note that $d=2^p -1$ gives $\beta = 1$.)
\end{proof}

\begin{theorem}\label{th:primes}
If $p$ is prime then $W(p+1,2)\ge p(2^p-1)$ by a deterministic-constructive proof.
\end{theorem}

\begin{proof}
Let $F=F_{2^p}$, the field on $2^p$ elements.
By Fact~\ref{fa:knownvec} $F$ is a vector space of dimension $p$
over $F_2$. Let $v_1,\ldots,v_p$ be a basis.
By Fact~\ref{fa:knowngen} there exists a generator $g$ such that
$$F-\{0\} = \{g,g^2,g^3,\ldots,g^{2^p-1}\}.$$

We express $g, g^2, \ldots, g^{2^p-1}, g^{2^p},\ldots,g^{p(2^p-1)}$ in terms of the basis.
This looks odd since $g=g^{2^p}$ so this list repeats itself; however, it will be useful.

For $1\le j\le p(2^p-1)$ and for $1\le i\le p$ let $a_{ij}\in \{0,1\}$ be such that

$$g^j = \sum_{i=1}^p a_{ij} v_i.$$

We now color $[p(2^p-1)]$. Let $j\in [p(2^p-1)]$. Color $j$ with $a_{1j}$.
That is, express $g^j$ in the basis $\{v_1,\ldots,v_p\}$ and color it with
the coefficient of $v_1$, which will be a 0 or 1.
We need to show that this is indeed a proper coloring.
Assume, by way of contradiction that the coloring is not proper.
Hence there is a monochromatic $(p+1)$-AP.
We denote it

$$a, a+d, \ldots, a+pd.$$

Since all of the numbers are in $[p(2^p-1)]$ we have $a+pd \le p(2^p-1)$ and thus $d \le 2^p-2$.
Therefore we get $g^d \ne 1$.

If we express any of

$$I=\{g^a, g^{a+d}, \ldots, g^{a+pd}\} =  \{ g^a, g^ag^d, g^ag^{2d}, \ldots, g^ag^{pd} \}$$

in terms of the basis they have the same coefficient for $v_1$.
Let $\alpha=g^a$ and $\beta=g^d\ne 1$.
Recall that, by Fact~\ref{fa:knownkey}, $\beta$ does not solve any degree $p-1$ 
polynomial with coefficients in $\{0,1\}$.

\medskip 

\noindent
{\bf Case 1:} The coefficient is 0. Then we have that all of the elements of $I$
lie in the $p-1$ dim space spanned by $\{v_2,\ldots,v_p\}$.
There are $p+1$ elements of $I$, so any $p$ of them are linearly dependent.
Hence $I'=\{ \alpha, \alpha\beta, \alpha\beta^2, \ldots, \alpha\beta^{p-1} \}$
is linearly dependent. So there exists $b_0,\ldots,b_{p-1} \in \bit$, not all 0,  such that 

$$\sum_{i=0}^{p-1} b_i \alpha \beta^{i} = 0$$

$$\sum_{i=0}^{p-1} b_i \beta^{i} = 0.$$

Therefore $\beta$ satisfies a polynomial of degree $\le p-1$ with coefficients in
$\{0,1\}$, contradicting Fact~\ref{fa:knownkey}.

\medskip 

\noindent
{\bf Case 2:} The coefficient is 1. 
Hence all of the elements of $I$, when expressed in the basis $\{v_1,\ldots,v_p\}$
have coefficient 1 for $v_1$. Take all of the elements of $I$ (except $\alpha$)
and subtract $\alpha$ from them. The set we obtain is

$$\{\alpha\beta-\alpha , \alpha\beta^2-\alpha, \ldots, \alpha\beta^p-\alpha \}=
  \{\alpha(\beta-1) , \alpha(\beta^2-1), \ldots, \alpha(\beta^p-1) \}.$$

\noindent
{\bf KEY:} All of these elements, when expressed in the basis, have coefficient 0 for $v_1$.
Hence we have $p$ elements in a $p-1$-dim vectors space.
Therefore they are linearly dependent.
So there exists $b_0,\ldots,b_{p-1} \in \bit$, not all 0, such that 

$$\sum_{i=0}^{p-1} b_i\alpha(\beta^i-1)= 0$$

$$\sum_{i=0}^{p-1} b_i(\beta^i-1)= 0.$$

Therefore $\beta$ satisfies a polynomial of degree $\le p-1$ over $F_2$. This contradicts Fact~\ref{fa:knownkey}.

We now express the above proof in terms of a deterministic construction.

\begin{enumerate}
\item
Input($p+1$).
\item
Find an irreducible polynomial $i(x)$ of degree $p$ over $F_2[x]$.
This gives a representation of $F_{2^p}$, namely \hbox{$F_2[x]/<i(x)>$.}
Note that $1,x,x^2,\ldots,x^{p-1}$ is a basis for $F_{2^p}$ over $F_2$.
Let $v_i=x^{i+1}$.
\item
Find $g$, a generator for $F_{2^p}$ viewed as a cyclic group.
\item
Express $g$, $g^2$, $\ldots$, $g^{p(2^p-1)}$ in terms of the basis.
For $1\le j\le p(2^p-1)$, for $1\le i\le p$ let $a_{ij}\in \{0,1\}$ be such that
$g^j = \sum_{i=1}^p a_{ij} v_i.$
\item
Let $j\in [p(2^p-1)]$. Color $j$ with $a_{1j}$.
\end{enumerate}

Steps 2 and 3 can be done in time polynomial in $2^p$ by Fact~\ref{fa:known}.
Step 4 can be done in time polynomial in $2^p$ using simple linear algebra.
Hence the entire algorithm takes time polynomial in $2^p$.
\end{proof}

\smallskip

Baker, Harman, and Pintz~\cite{BHP} (see~\cite{ConPrimesSurvey} for
a survey) showed that, for all but a finite number of $k$, there is a prime between
$k$ and $k-k^{0.525}$. Hence we have the following corollary.

\begin{corollary}
For all but a finite number of $k$, 
$$W(k,2) \ge (k-k^{0.525})(2^{k-k^{0.525}}-1).$$
(We do not claim this proof is deterministic-constructive or randomized-constructive.)
\end{corollary}

\begin{proof}
Given $k$ let $p$ be the primes such that 
$k-k^{0.525}\le p \le k$.
By Theorem~\ref{th:primes} $W(p+1,2)\ge p(2^p-1)$
Hence
$$W(k,2) \ge W(p+1,2) \ge p(2^p-1) \ge (k-k^{0.525})(2^{k-k^{0.525}}-1).$$
\end{proof}

\section{A Bit of Kolmogorov Theory}\label{se:kolm}

We will need some Kolmogorov theory for the next section and thus give a short introduction here.
For a fuller and more rigorous account of Kolmogorov Theory see the book by Li and Vitanyi~\cite{LV}.

What makes a string random? Consider the string $x=0^n$. This string does not seem that random but how can we pin that down?
Note that $x$ is of length $n$ but can be easily produced by a program of length $\lg(n)+O(1)$ like this:

{
\obeylines{
	\qquad FOR $x=1$ to $n$, PRINT(0)
}
}

\bigskip

\noindent By contrast consider the following string

$$x=0110100101010010101011111100001110010101$$

\noindent
which we obtained by flipping a coin 40 times.
It can be produced by the following program.

{
\obeylines{
	\qquad PRINT(0110100101010010101011111100001110010101)
}
}

Note that this program is of length roughly $|x|$. 
There does not seem to be a shorter program to produce $x$.
The string $x$ seems random in that there is no pattern in $x$
which would lead to a shorter program to print $x$ than the one above. 
Informally a string $x$ looks random, if the shortest program to print out $x$ has length roughly $|x|$.
We formalize this.

\begin{definition}
Fix a programming language $L$ that is Turing complete.
Let $x\in \bits n$ (think of $n$ as large)
and $y\in \bits m$ (think of $m$ as small).
$K_L(x|y)$ is the 
length of the shortest
program $P$ in $L$ such that $P(y)$ has output $x$.
\end{definition}

\begin{fact}
If $L_1$ and $L_2$ are Turing complete programming languages 
then there is a program that translates one to the other.
This program is of constant size. Hence there is a constant $\alpha\in\nat$ such that
$|K_{L_1}(x|y) - K_{L_2}(x|y)| \le \alpha$.
Therefore $K_L(x|y)$  is independent of $L$ up to an additive constant factor.
Hence we will drop the $L$ and always include an $O(1)$ or $\Omega(1)$
term as is appropriate.
\end{fact}

\begin{definition} \label{def:kolrandom} 
A string $x$ is {\it Kolmogorov random relative to $y$} if $K(x|y) \ge |x|+\Omega(1)$.
\end{definition}

\begin{fact}
By comparing the number of strings of length $n$ to the
number of descriptions of length smaller than $n$ we conclude 
that most strings are Kolmogorov random.
Hence if you find that a randomized algorithm works
well when you use a Kolmogorov random string for
the random bits, then it works well for most strings.
We will assume that at least $2/3$ of all strings of length
$n$ are Kolmogorov random; however, there are really far more.
\end{fact}

\section{A Randomized-Constructive Lower Bound via the Lov{\'a}sz Local Lemma}\label{sec:rconstructive}

We use the following lemma both in this section and the next section.
The bulk of this lemma is an exercise from Knuth~\cite{knuth1969}; however,
we include the proof for completeness.

\begin{lemma}\label{le:trees}
Let $m\in\nat$ and $T,T_1,\ldots,T_m$ be infinite rooted trees with each node having exactly $x$ ordered children. 
\begin{enumerate}
\item
There are at most $\binom{xs}{s} \leq (ex)^s$ subtrees of $T$ that include the root and consist of exactly $s$ non-root nodes.
\item
Let $\cal F$ be the set of all forests $F$ consisting of at most $m$ trees, 
such that each tree is a subtree of a different $T_i$, and such that the total number of nodes in $F$ is $s$. 
$\cal F$ consists of at most $2^m (ex)^s$ forests.
\end{enumerate}
\end{lemma}

\begin{proof}

\noindent
1.) Given a subtree of $T$ with $s$ non-root nodes, record an ordered DFS traversal using a zero to denote that a potential child is not there and a one for every forward step along an existing child. Stop the traversal at the last non-root node without recording zeros for its children. There are $x$ (potential) children each for both the root and each but the last of the $s$ non-root nodes; each of these $sx$ children appears at most once in the traversal. Therefore a string of length at most $sx$ is recorded. The string has furthermore exactly $s$ ones, one for each non-root node. Note that any two different subtrees of $T$ correspond to two different strings. We thus have an injection from the specified subtrees into the set of zero-one strings of length at most $xs$ with exactly $s$ ones. There are exactly $\binom{sx}{s}$ such strings and therefore also at most this many subtrees of $T$ with $s$ non-root nodes.
The inequality $\binom{sx}{s} < (xe)^s$ follows from Stirling's formula.

\bigskip

\noindent
2.) Let $T'$ be the ordered tree that has a root $r$ of degree $m$
and at the $i^{th}$ child of $r$ attached $T_i$ (so the $i^{th}$ node on the second
level is the root of $T_i$). There is a straight forward bijection between
forests in $\cal F$ and subtrees of $T'$ with $s$ non-root nodes.

We describe such a subtree of $T'$ by a subset of $[m]$ to specify the
children of $r$ that are not used and by a zero-one string of 
length at most $xs$ with exactly $s$ ones corresponding to a 
DFS-traversal of the remaining tree in the same manner as above.
Each forest in $\cal F$ can be uniquely described
in such a manner (but not all those descriptions correspond to a valid tree). The total number of those descriptions and therefore also the total number of forests in $\cal F$ is at most $2^m (xe)^s$.
\end{proof}

\begin{theorem}\label{th:main}
$W(k,2)\ge \frac{2^{(k-1)}}{4k}$ by a randomized-constructive proof.
\end{theorem}

\begin{proof}
Let $n=\frac{2^{(k-1)}}{4k}$. 
We present a randomized algorithm to find a 2-coloring of $[n]$. Let $E_1,\ldots,E_m$ be the $k$-AP's 
of $[n]$ listed in lexicographic order. 
By Lemma \ref{le:numof}, $m=O(n^2/k)$. 

We present a simple algorithm with a parameter $s$, which we will determine later.

\noindent
{\bf MAIN ALGORITHM}

\begin{enumerate}
\item
Color $[n]$ using $n$ random bits
\item
$NUMCALLS=0$ (this will be the number of calls to $FIX$).
\item
For $i=1$ to $m$ if $E_i$ is monochromatic then $FIX(E_i)$.
\item
Output the coloring.
\end{enumerate}

\noindent
{\bf END OF MAIN ALGORITHM}

\noindent
{\bf FIX ALGORITHM}

\noindent
$FIX(E)$
\begin{enumerate}
\item
$NUMCALLS =NUMCALLS+1$.
\item
If $NUMCALLS=s$ then STOP and output {\bf I HAVE FAILED}.
\item
Recolor $E$ randomly (this takes $k$ random bits).
\item
While there exists a monochromatic $k$-AP that intersects $E$ 
let $E'$ be the lexicographic smallest such $k$-AP and call $FIX(E')$.
\end{enumerate}

\noindent
{\bf END OF FIX ALGORITHM}

\bigskip

We leave the following easy claims to the reader:

\noindent
{\bf Claim 1:}
For all calls to $FIX$ that terminate the following holds:
all of the $k$-AP's that were not monochromatic
before the call are not monochromatic after the call.

\noindent
{\bf Claim 2:}
If the algorithm outputs a coloring then it is a proper coloring.

\bigskip

We find a value for the parameter $s$ such that $s$ is polynomial in $n$ and $k$ and such that the 
probability of the algorithm's success is at least  $2/3$. 
With parameter $s$ the algorithm uses at most $n+sk$ random bits.
We can think of the algorithm as a deterministic one which takes an additional
$n+sk$ bit string as input to use in place of the random bits.
Let $z=z_0z_1\cdots z_s$ denote that string. The first $n$ bits are used for the
initial color assignment, and the remaining bits are used for the reassignments as
needed. 

Let $z=z_0z_1\cdots z_s$ be a Kolmogorov random string relative to $k,n$.
We will show that if the algorithm is run with $z$ supplying the random bits
then the result will be a proper coloring of $[n]$.
Since over 2/3 of all strings of length $n+sk$ are Kolmogorov random relative to $k,n$ 
this will prove that the algorithm succeeds with probability $\ge 2/3$.

Assume, by way of contradiction, 
that the algorithm goes through $s$ calls to $FIX$. 
We will pick a value of $s$ so that this leads to a contradiction.

\begin{definition}
The {\it FIX-FOREST} is the forest of calls to $FIX$.
We take the children of a node to be ordered in the same order the
procedure FIX was called.
The nodes are labeled by what monochromatic $k$-AP they were called with and
what color (a bit) the $k$-AP was before the call.
\end{definition}

\begin{definition}
For $1\le i\le m$ we define a tree $T_i$ as follows.
\begin{itemize}
\item
The root is labeled with $E_i$ (the $i^{th}$ $k$-AP in lexicographic order).
\item
If a node is labeled with a $k$-AP $E$ then the children
are the $k$-AP's that intersect $E$ in lexicographic order.
\end{itemize}
\end{definition}

Putting all this together we get:

\begin{fact}\label{fa:FIX}~
\begin{enumerate}
\item
By Lemma~\ref{le:numof} every node of $T_i$ has at most $kn$ children.
\item
By Claim 1 the FIX-FOREST has less than $n^2/k$ trees
\item
All trees in the forest are subtrees of different subtrees $T_i$.
\end{enumerate}
\end{fact}

This makes it possible to apply Lemma~\ref{le:trees} 
and obtain that the number of different FIX-FOREST structures is
at most $2^{\frac{n^2}{k}} (kn)^s$. From this we get that, given $n$ and $k$, each
FIX-FOREST can be described by $\frac{n^2}{k} + s\lg(kn) + O(1)$ bits 
for its structure and another $s$ bits for the color labels. Now let
$w$ be the coloring after $s$ calls to $FIX$ are performed. Note
that $w$ can be described with $n$ bits. The next claim shows
that taking all these descriptions it is possible to reconstruct
the Kolmogorov random string $z$.

\bigskip 

\noindent
{\bf Claim 3:} Given $n,k$ the FIX FOREST and $w$ one can recover $z$.

\noindent
{\bf Proof of Claim 3}

From the FIX FOREST we can obtain:

\begin{itemize}
\item
a description of the $k$-AP $a_i$ that the $i^{th}$ call was made on.
\item
the color $c_i$ of $a_i$ when the $i^{th}$ call to FIX was made.
\end{itemize}

We recover the $z$'s in three phases.
%(We do a full example in the appendix.)

\medskip 

\noindent{\bf Phase I (just use the $a_i$'s but not the $c_i$'s):}
Simulate the Coloring Algorithm using 
the symbols $z_i^j$ where 
($0\le i\le s$, if $i=1$ then $1\le j\le n$, if $i\ge 2$ then $1\le j\le k$)
to represent the $j$th bit of $z_i$.
Note that we do not know the actual colors so we really do use (say) $z_3^4$
and not RED (or more formally 0 or 1).
Since we have $a_i$ we can  (and do)
keep track of the coloring of $[n]$ after each call to $FIX$,
in terms of the symbols $z_i^j$.
This creates a table of $z_i^j$'s.

For example, if $n=15$ then the first row will be:

\[
\begin{array}{c||c|c|c|c|c|c|c|c|c|c|c|c|c|c|c|}
1&z_0^1&z_0^2&z_0^3&z_0^4&z_0^5&z_0^6&z_0^7&z_0^8&z_0^9&z_0^{10}&z_0^{11}&z_0^{12}&z_0^{13}&z_0^{14}&z_0^{15}\cr
\end{array}
\]

If $k=4$ $a_1=(4,7,10,13)$, i.e., the first call to FIX was to $(4,7,10,13)$, then the second row will be

\[
\begin{array}{c||c|c|c|c|c|c|c|c|c|c|c|c|c|c|c|}
2&z_0^1&z_0^2&z_0^3&z_1^1&z_0^5&z_0^6&z_1^2&z_0^8&z_0^9&z_1^{3}&z_0^{11}&z_0^{12}&z_1^{4}&z_0^{14}&z_0^{15}\cr
\end{array}
\]

\smallskip

\noindent{\bf Phase II (use the $c_i$'s to determine $z_i^j$'s):}
For all $1\le i\le s$, right before the $i^{th}$ call to $FIX$,
$k$-AP $a_i$ was monochromatic; all $k$ vertices of $a_i$ were colored $c_i$.
For $1\le j\le i-1$  let $V_j$ be the vertices of $a_i$
that were most recently colored by $z_j$ (note that $V_j$ could be empty).
By Phase I we know which bits of $z_j$ colored which vertices of $V_j$.
We now know that those bits are $c_i$.
For each $i$ we have recovered $k$ bits of $z$.
Since there are $s$ calls to $FIX$ this phase recovers $sk$ bits.

\smallskip
\noindent{\bf Phase III (use $w$):}
For each $x\in [n]$ there is an $i,j$ so that $x$ was colored
$z_i^j$ and never recolored. We now know the $z_i^j=w^x$.
This phase recovers $n$ bits of $z$.

The phases all together recover $n+sk$ bits of $z$. Since $|z|=n+sk$
all of $z$ is recovered.

\noindent
{\bf End of Proof of Claim 3}

\medskip 

By Claim 3, $z$ can be described using $n,k$ the FIX FOREST and $w$.
Since $w$ can be described by $n$ bits and since Fact~\ref{fa:FIX}.4 tells us that the FIX FOREST can be described
with $s + \frac{n^2}{k} + s\lg(kn) + O(1)$ bits we get a description of $z$ of size $s + \frac{n^2}{k} + s\lg(kn) + O(1) + n$.

On the other hand we assumed $z$ to be Kolmogorov random relative to $k,n$ which implies 
that any description of $z$ has to have length at least $n+sk+O(1)$. Hence

$$s + \frac{n^2}{k} + s\lg(kn)+O(1) + n \ge n+sk$$

$$\frac{n^2}{k} + O(1)  \ge sk-s-s\lg(kn)$$

%$$\frac{n^2}{k} + O(1)  \ge s(k-1-\lg(kn))$$

$$s \le \frac{\frac{n^2}{k} + O(1)}{k-1-\lg(kn)} < n^2/k^2 + O(1)$$

Now choosing $s \ge \frac{n^2}{k^2} + O(1)$ leads to the desired contradiction.
\end{proof}

\section{A Deterministic-Constructive Lower Bound by\\ Derandomizing the Lov{\'a}sz Local Lemma}

\begin{theorem}\label{thm:pconstructive}
Fix $\eps>0$.  $W(k,2) \ge \frac{2^{(k-1) (1 - \eps)}}{ek}$ by a deterministic-constructive proof.
\end{theorem}

\begin{proof}
Let  $n=\frac{2^{(k-1)(1-\eps)}}{ek}$. (We assume $n$ is an integer; the modifications to make this rigorous are easy but cumbersome.)
We will present a deterministic algorithm that always produces
a proper 2-coloring of $[n]$ and runs in time $n^{O(\eps^{-1.01})}$ which is polynomial as long as $\eps$ is any fixed constant. The algorithm proceeds in stages.

Let $t$ be a parameter to be named later. It will be $O(\eps^{-1.01})$.

\bigskip

\noindent
{\bf Stage 1: List out Trees}

We create by exhaustive enumeration a list of the following set of trees $Y$:\\
For each $k$-AP $E$, for each subset $S$ of $E$, take all possible labeled trees that satisfy the following properties:

\begin{enumerate}
\item
the root is labeled with $S$,
\item
each non-root node is labeled with a $k$-AP of $[n]$,
\item \label{prop:neighbors}
the labels of each child of a node $B$ share a number with the label of $B$,
\item \label{prop:levelsdisjoint}
the labels of nodes on the same level are disjoint,
\item
there are between $\thresh$ and $2\thresh$ non-root nodes.
\end{enumerate}

\begin{verbatim}
EXAMPLE: A few trees in Y for n=7,k=3,t=1.5

23            246                35
 |             |                 |  \
234           567               123  567
               |		
23            135
 |
357            34                4
               |  \	             |
 3            123  456          147
 |                               |
123            14               357
 |             |  \
123           135  246

\end{verbatim}

To bound the running time of this and the next stage we need to check 
that the number of trees in $Y$ is always polynomial in $n$. 
For this recall that by Lemma~\ref{le:numof} there are at 
most $n^2/k$ different $k$-AP's which gives us at most $2^k n^2/k$ possible roots. 
If the root is fixed then by property \ref{prop:neighbors} and again Lemma~\ref{le:numof} 
we know that each tree in $Y$ is a subtree of the infinite tree in which each $k$-AP has as children the $< nk$ $k$-APs it intersects with. 
Using Lemma~\ref{le:trees}.2 with $x=kn$ and $s=t$ we obtain that there are at most 
$$(nke)^{2\thresh}$$
such subtrees of size $2\thresh$. Therefore the total number of trees $|Y|$ is at most 
$$\frac{2^k n^2}{k} (nke)^{2\thresh} \thresh \le n^3 \cdot (n^2)^{2\thresh} n \le n^{O(\thresh)} = n^{O(\eps^{-1.01})}.$$
Hence this stage of the algorithm runs in polynomial time for any fixed constant $\eps>0$. 

\bigskip

\noindent
{\bf Stage 2: Creation of a good table}

Similar to the proof of Theorem \ref{th:main} we create a table with a sequence of colors for every number. A {\it table} is a map $T: [n] \times [\thresh] \into \{0,1\}$ in which we view each row as a sequence of colors for its (column)-number. We will be looking at colorings of the numbers on the nodes of the tree that is guided by a table $T$. $T(x,t)$ will tell us how to color the number $x$ the $t^{th}$ time we look for a color of $x$ when we process the tree level-by-level from leaf-to-root. More formally we assign each number $x$ in the label of a node $v \in \tau$ the color 
$$T(x,1 + \hbox{ number of nodes below $v$ whose label contain $x$}).$$
Given a tree $\tau$ and a coloring of its numbers guided by table $T$ say $\tau$ is \emph{consistent} with $T$ iff all labels of $\tau$ are colored monochromatically. 

Note that because of property \ref{prop:levelsdisjoint} each color in the table $T$ gets used only once during this process. Thus if all colors in $T$ are chosen independently at random each label of a node gets monochromatic independently with probability $2^{-(k-1)}$. The probability for a tree in $Y$ to be consistent is thus at most $2^{-(k-1)i}$ where $i$ is the number of non-root nodes. Having this and computing $(\frac{n^2}{k} 2^k (nke)^i)$ as an upper bound for the number of trees with $i$ non-root nodes as in Stage $1$ we get that the expectation of the number of consistent trees $X$ is at most

$$E[X] \leq \sum_{i=\thresh}^{2\thresh} (n^2 2^k (nke)^i) 2^{-(k-1)i} = n^2 2^k \sum_{i=\thresh}^{2\thresh} 2^{-(k-1) \eps i} < n^2 2^k \thresh 2^{-(k-1) \eps \thresh} \leq 2^{O(k)} t 2^{-k \eps t}.$$

Thus the expectation is $< 1/3$ if $\thresh = O(\eps^{-1.01})$ is picked large enough. 
Markov's inequality proves that with probability at least $\frac{2}{3}$ no tree in $Y$ is consistent with a randomly chosen table. We efficiently construct such a table using the method of conditional expectations in the following algorithm:

\bigskip

\noindent
{\bf TABLE CREATION ALGORITHM}

\begin{itemize}
\item
For all $x=1$ to $n$, For all $y=0$ to $2t$
\begin{itemize}
\item Set $T(x,y) = 1/2$
\end{itemize}
\item
For all $x=1$ to $n$, For all $y=0$ to $2t$
\begin{itemize}
\item 
Set $T(x,y) = 0$\\
For all $\tau \in Y$
\begin{itemize}
	\item Compute $p_{\tau,0} = \prod_{v \in \tau} (\prod_{i \in \hbox{label}(v)} \hbox{color}(i) + \prod_{i \in \hbox{label}(v)} (1 - \hbox{color}(i)))$
	\item (Here $\hbox{color}(i)$ corresponds to the entry from $T$ that is assigned to this number $i$ in the label of node $v$ when the coloring of $\tau$ is guided by $T$. Note that $p_{\tau,0}$ corresponds exactly to the probability that every node-label in $\tau$ becomes monochromatic if colors are filled in from the table $T$ into $\tau$ level-by-level from leaf-to-root while choosing a random color instead of any $1/2$.)
\end{itemize}
Let $E_0 = \sum_{\tau \in Y} p_{\tau,0}$
\item 
set $T(x,y)=1$ and compute all $p_{\tau,1}$ and $E_1$ similarly to the last step
\item 
if $E_0 < E_1$ than $T(x,y)=0$ else $T(x,y)=1$ in order to minimize the expectation.
\end{itemize}
\end{itemize}

\noindent
{\bf END OF TABLE CREATION ALGORITHM}

\bigskip

For the analysis of the table creation we see that $E_0$ and $E_1$ are exactly the expected number of consistent trees in $Y$ if we set $T(x,y)$ to $0$ or $1$ respectively. Because of our choice of $t$ from above we get that in the beginning this expectation is $E = \frac{E_0 + E_1}{2}<2/3$. By always choosing the color that minimizes this expectation the invariant $\frac{E_0 + E_1}{2}<2/3$ is preserved throughout the algorithm. When finally all entries of $T$ are chosen, no randomness remains and the invariant implies that no tree with properties 1-5 is consistent with $T$. This stage of the algorithm takes $O(4tk|Y|)$ time for each of the $2tn$ iterations and therefore runs in time polynomial in $n$.

\bigskip

\noindent
{\bf Stage 3: Run a Recoloring Algorithm using Colors from the Table}

\begin{enumerate}
\item
Initially color $[n]$ using the first column of $T$.
\item
WHILE there is a monochromatic $k$-AP $E$\\
   \mbox{}\hspace{0.5cm} recolor the numbers in $E$ using for each number its next unused color from $T$
\end{enumerate}

This completes the algorithm. In the rest of this section we show that the algorithm terminates without requesting more than $\thresh$ colors for one number which will be enough to argue a quick termination. Note that because of the termination condition of the algorithm no proof of correctness is needed. 

\bigskip

\noindent
{\bf Claim:}
Each number gets recolored at most $\thresh$ times.

\noindent
{\bf Proof of Claim}

Lets look at the sequence of $k$-APs as picked by the algorithm. 
For each $k$-AP $E$ in this sequence and each  subset $S$ of $E$ we construct a tree labeled by subsets of 
$[n]$ by starting with a root with label $S$. Going back in the sequence we iteratively take the next $k$-AP $B$ and if there is a node in the tree whose label shares a number with $B$ we create a new node with label $B$ and attach it to the lowest such node breaking ties arbitrarily.

Let $Z$ be the set of trees that can be constructed from the run of the algorithm using the table $T$. We prove the claim in the following two steps:
\begin{enumerate}
\item
All trees in $Z$ are consistent with the table $T$.
\item
If a number got recolored more than $\thresh$ times then there exists a tree $\tau \in Y \cap Z$ which leads to the desired contradiction.
\end{enumerate}

\noindent
{\bf All trees in $Z$ are consistent with the table $T$:}

We want to argue that the colors that gets filled from $T$ into a $k$-AP $E$ when consistency is checked are exactly the same entries in $T$ that the algorithm sees before it recolors this $k$-AP $E$, i.e., both are monochromatic. Focusing on one number $x \in E$ we directly see that the entry from $T$ that is used to recolor $x$ is the entry with the number $i$ from the column in $T$ that belongs to $x$, where $i$ is the number of times a color for $x$ was needed before which is exactly one plus the number of $k$-APs containing $x$ that got recolored before. 
Note that all these $k$-APs appear in a tree below $E$ which is the reason why when consistency is checked for also exactly the entry $i$ is filled into $x$ (see definition of consistency). This proves that any tree that got created from a run with table $T$ is consistent with $T$.

\medskip 

\noindent
{\bf If a number got recolored more than $\thresh$ times then a tree $\tau \in Y$ is constructed:}

For sake of contradiction we assume that a number got recolored more than $\thresh$ times and argue that in this case a tree $\tau \in Y$ gets constructed. Note that by construction all trees fulfill the properties $1$-$4$. Hence it remains that a tree of size between $\thresh$ and $2\thresh$ is generated from the trace. For this let $\tau$ be the the smallest tree in $Z$ of size $s \geq \thresh$. Such a tree exists because generating a tree from the 
$\thresh^{\rm th}$ time a number got recolored produces a tree of size at least $\thresh$. If the label $S$ of the root of $\tau$ consists of just one number then because of property 3 and 4 it has only one child and the tree generated choosing this child as a root label has size $s-1$. Otherwise take one number $x \in S$ and look at the trees generated with $\{x\}$ and $S - \{x\}$ as a root label. One of them has size at least $s/2$ since each node in the tree generated by $S$ appears in at least one of the new trees. In either case the minimality of $\tau$ -- that the remaining tree of size either $s-1$ or $s/2$ has to be smaller than $\thresh$ -- implies $s\leq 2\thresh$. This shows that the tree $\tau$ that is constructed from the trace fulfills all 5 properties and is therefore a tree from $Y$ that is consistent with the table $T$. This is a contradiction to the way we constructed the table $T$ in stage 2.

\noindent
{\bf End of Proof of Claim}

\bigskip

It is easy to see that with the guarantee given by this claim the algorithm runs for at most $O(\thresh n)$ time in this stage and terminates with a proper coloring. With all previous stages running in time polynomial in $n$ the entire algorithm does so and thus fulfills the properties of a deterministic-constructive proof, finishing the proof of Theorem \ref{thm:pconstructive}.  
\end{proof}

\begin{corollary}\label{co:constLLL}
$W(k,2) \ge \frac{2^{(k-1)}}{ek} - 1$ by a randomized-constructive proof.
\end{corollary}

\begin{proof}
The algorithm used to achieve this bound is simply stage 3 of the algorithm above but instead of using the colors from a carefully prepared table $T$ an independent uniformly random color is chosen each time a new color is needed. If more than $t$ new colors are requested for any number the algorithm stops and reports its failure. 
This is the randomized algorithm of Moser and Tardos~\cite{mosertardos} which is very similar and actually encompasses 
Moser's algorithm given in Section \ref{sec:rconstructive}. 
For its analysis we note that the only reason why Theorem \ref{thm:pconstructive} does not give us the bound of this theorem is because we can not make $\eps$ smaller with $k$. The reason for this is that the running time of stage 1 and 2 is $n^{O(\eps^{-1.01})}$ time which forces $\eps$ to be a fixed constant. Since the randomized algorithm only runs stage 3 we can choose $\eps = \Theta(\frac{1}{n'k})$ where $n' = \frac{2^{(k-1)}}{ek}$ such that $2^{-(k-1)\eps} \geq e^{-n'} \geq (1 - 1/n')$. This still keeps the running time of stage 3 to be polynomial, more specifically $O(tn)=O(n \eps^{1.01})=O(n (n'k)^{1.01})$. The success probability comes directly from the analysis for stage 2. There is already stated that the probability for random colors to form a good table is at least $2/3$. Thus also the success probability of the described algorithm to reports a proper 2-coloring is as required by the definition of randomized-constructive. With such a small $\eps$ the lower bound implied by this randomized algorithm now becomes $W(k,2) \ge \frac{2^{(k-1)(1-\eps)}}{ek} \geq \frac{2^{(k-1)}}{ek} (1 - 1/n') = \frac{2^{(k-1)}}{ek} - 1$ as desired. 
\end{proof}

\section{Acknowledgement}

We would like to thank Robin Moser for his brilliant talk at STOC 2009 which inspired this paper. We would also like to thank Thomas Dubois, Mohammad Hajiaghayi and Larry Washington for proofreading and helpful comments. Last but not least we would like to thank the anonymous reviewer(s) for catching several minor mistakes and for helpful suggestions which greatly improved the presentation of this paper.

\end{document}